\documentclass[12pt]{article}
\usepackage{amsmath, amssymb, amsthm}
\usepackage[colorlinks=true]{hyperref}

\theoremstyle{plain}
\newtheorem{theorem}{Theorem}
\newtheorem{corollary}[theorem]{Corollary}
\newtheorem{lemma}[theorem]{Lemma}

\theoremstyle{definition}
\newtheorem{definition}[theorem]{Definition}

\theoremstyle{remark}
\newtheorem{remark}[theorem]{Remark}

\newcommand{\N}{\mathbb{N}}
\newcommand{\Z}{\mathbb{Z}}

\newcommand{\R}{\mathbb{R}}

\newcommand{\ds}{\displaystyle}

\allowdisplaybreaks[1]

\date{}

\begin{document}

\title{A Proof of Symmetry of the Power Sum Polynomials using a Novel Bernoulli Numbers Identity}
\markright{Symmetry of the Power Sum Polynomials}

\author{Nicholas J. Newsome, Maria S. Nogin,\\ and Adnan H. Sabuwala\\Department of Mathematics\\California State University, Fresno\\Fresno, CA 93740 \\ USA\\\href{mailto:mrgoof@mail.fresnostate.edu}{mrgoof@mail.fresnostate.edu}\\\href{mailto:mnogin@csufresno.edu}{mnogin@csufresno.edu}\\\href{mailto:asabuwala@csufresno.edu}{asabuwala@csufresno.edu}}

\maketitle

\begin{abstract}
The problem of finding formulas for sums of powers of natural numbers has been of interest to mathematicians for many centuries.  
Among these is Faulhaber's well-known formula expressing the power sums as polynomials whose coefficients involve Bernoulli numbers. In this paper we give an elementary proof that the sum of $p$-th powers of the first $n$ natural numbers can be expressed as a polynomial in $n$ of degree $p+1$. We also prove a novel identity involving Bernoulli numbers and use it to show symmetry of this polynomial. 
\end{abstract}

\section{Introduction.} 

In an introductory calculus class, undergraduate students usually encounter the following summation formulas: 
\begin{equation}
\label{power-sum-formulas}
\begin{array}{l@{\ }c@{\ }l}
\ds \sum_{k=1}^n k & = & 1+2+\dots+n = \ds \frac{n(n+1)}{2} \\ 
\ds \sum_{k=1}^n k^2 & = & 1^2+2^2+\dots+n^2 = 
\ds \frac{n(n+1)(2n+1)}{6} \\ 
\ds \sum_{k=1}^n k^3 & = & 1^3+2^3+\dots+n^3 = 
\ds \frac{n^2(n+1)^2}{4} \\ 
\end{array}
\end{equation}
for any positive integer $n$. Typically, these are proved by mathematical induction and are used in the Riemann sum evaluation of definite integrals.

The problem of finding formulas for sums of powers of integers has captivated mathematicians for many centuries \cite{Beery}. The Pythagoreans were the first to discover the formula for the sum of the first powers with their pebble experiments. Archimedes and Aryabhata are credited for a geometric proof of the remaining two formulas above, respectively. These formulas were first introduced in a generalizable form by Harriot. Faulhaber provided formulas for power sums up to the $17^{\text{th}}$ power, but did not make it clear how to generalize them. Later, Fermat, Pascal, and Bernoulli discovered and presented succinct formulas for representing these sums. 

Ever since generalized formulas for the powers sums,   
$\displaystyle S_p(n)=\sum_{k=1}^{n}k^p$, 
have been established, their various representations and number-theoretic properties have been studied \cite{Knuth, MS}. Faulhaber's formula expresses the power sums as polynomials whose coefficients involve Bernoulli numbers \cite[p.\ 107]{Conway} that have a rich history of their own. Until close to the $21^{\text{st}}$ century, very few recurrences of Bernoulli numbers were known. Such recurrences in Bernoulli numbers reveal important aspects and properties of these numbers that can be used to simplify proofs of known identities \cite{Howard}. Namias \cite{Namias} derived recurrences based on Gauss's multiplication formula and a generalization of those formulas was posed as a problem in the \emph{The Amer. Math. Monthly} \cite{Belinfante}. Tuenter \cite{Tuenter}  proved a relation of symmetry between the power sum polynomials and Bernoulli numbers which can be used to generalize Namias's recurrence formulas. The power sum polynomials can be expressed in terms of Bernoulli polynomials \cite{Apos}: 
$$S_p(n)=\frac{1}{p+1}(B_{p+1}(n+1)-B_{p+1}(1)).$$ 
Lehmer \cite{Lehmer} proved symmetry of Bernoulli polynomials 
$$B_m(1-x)=(-1)^mB_m(x)$$
using the Fourier series representation. This implies symmetry of the power sum polynomials
$$S_p(-(n+1))=(-1)^{p+1}S_p(n)$$
although this result has not appeared in the literature. In this paper, we present an alternate proof of the above symmetry of the power sum polynomials using a novel identity involving Bernoulli numbers.

This paper is organized into five sections. In section \ref{recursive-definition-section}, we provide a recursive definition for $S_p(n)$ which is used to prove that the sum of $p$-th powers of natural numbers is a polynomial of degree $p+1$. In section \ref{Bernoulli-section}, we prove a novel identity for the Bernoulli numbers. In section \ref{Spn-symmetry-section}, we use this identity to prove the symmetry of $S_p(n)$. Lastly, in section \ref{open-problems-section} we pose some open problems about the roots of $S_p(n)$. 

\section{A recursive definition for \texorpdfstring{$S_p(n)$}{Spn}.}
\label{recursive-definition-section}

\begin{definition}
\label{Spn-definition}
For $n\in\R$, let $S_1(n)=\ds\frac{n(n+1)}{2}$. 

For $p \ge 2$ and $n \in \R$, define
$$S_p(n)=\frac{1}{p+1}\left((n+1)((n+1)^p-1)-\sum_{i=1}^{p-1}\binom{p+1}{i}S_i(n)\right).$$
\end{definition}

\begin{remark}
For each $p\in\N$, $S_p(n)$ is a polynomial of degree $p+1$. 
\end{remark}

\begin{lemma}
\label{first-lemma}
For $p \in \N$ and $n \in \R$, 
$$(n+1)\left((n+1)^p-1\right)=n(n+1)^p+(n+1)\left((n+1)^{p-1}-1\right).$$
\end{lemma}
\begin{proof}
Indeed, 
\begin{align*}
(n+1)\left((n+1)^p-1\right) & =(n+1)(n+1)^p-(n+1)\\
&= n(n+1)^p+(n+1)^p-(n+1)\\
& = n(n+1)^p+(n+1)\left((n+1)^{p-1}-1\right). 
\end{align*}
\end{proof}

\begin{lemma}
\label{second-lemma}
For $p \in \N$ and $n \in \R$, 
$$\sum_{i=1}^{p} \binom{p+1}{i-1}n^i=n\left((n+1)^{p+1}-n^p(p+n+1)\right).$$
\end{lemma}

\begin{proof}
We will use the convention that $\ds \binom{k}{-1}=0$ for all $k \in \N$.

We proceed by induction on $p$. If $p=1$, then
$$\sum_{i=1}^{1}\binom{2}{i-1}n^i=\binom{2}{0}n^1=n=n\left((n+1)^{2}-n(1+n+1)\right),$$
so the statement holds.

Now assume that the statement holds for some $p\ge 1$. Then 
\begin{align*}
\sum_{i=1}^{p+1}\binom{p+2}{i-1}n^i &= \sum_{i=1}^{p+1}\left(\binom{p+1}{i-1}+\binom{p+1}{i-2}\right)n^i\\
&= \sum_{i=1}^{p+1}\binom{p+1}{i-1}n^i+\sum_{i=1}^{p+1}\binom{p+1}{i-2}n^i\\
&= \sum_{i=1}^{p}\binom{p+1}{i-1}n^i+\binom{p+1}{p}n^{p+1}+\underbrace{\sum_{i=2}^{p+1}\binom{p+1}{i-2}n^i}_{\text{by convention}}\\
&=n\left((n+1)^{p+1}-n^p(p+n+1)+n^p(p+1)\right)\\
&\qquad +\sum_{i=1}^{p}\binom{p+1}{i-1}n^{i+1}\\
&=n\left((n+1)^{p+1}-n^{p+1}+\sum_{i=1}^{p}\binom{p+1}{i-1}n^{i}\right)\\
&=n\left((n+1)^{p+1}-n^{p+1}+n((n+1)^{p+1}-n^p(p+n+1))\right)\\
&=n\left((n+1)^{p+1}-n^{p+1}+n(n+1)^{p+1}-n^{p+1}(p+n+1)\right)\\
&= n\left((n+1)^{p+2}-n^{p+1}((p+1)+n+1)\right).
\end{align*}
\end{proof}

\begin{theorem}
\label{power-sum-theorem} 
For $p,n\in\N$, $S_p(n)=\ds \sum_{k=1}^{n}k^p$. 
\end{theorem}

\begin{proof}
We use strong induction on $p$. 

For $p=1$, $\ds S_1(n)=\frac{n(n+1)}{2} = \sum_{k=1}^n k$.

For $p=2$, 

\begin{align*}
S_2(n) &= \frac{1}{3}\left((n+1)\left((n+1)^2-1\right)-\binom{3}{1}S_1(n)\right)\\ 
&= \frac{1}{3}\left((n+1)(n^2+2n)-3\left(\frac{n(n+1)}{2}\right)\right)\\
&=\frac{1}{3}\cdot\frac{n(n+1)}{2}\left(2(n+2)-3\right)\\
&= \frac{n(n+1)(2n+1)}{6}\\
&= \sum_{k=1}^{n}k^2. 
\end{align*}

Now assume that for some $p\ge 2$, $S_i(n)=\ds \sum_{k=1}^{n}k^i$ holds for all $1 \le i \le p$. Then 
\begin{align*}
(p+2)S_{p+1}(n)&=(n+1)\left((n+1)^{p+1}-1\right)-\sum_{i=1}^{p}\binom{p+2}{i}S_i(n)\\ 
&= (n+1)\left((n+1)^{p+1}-1\right)-\sum_{i=1}^{p}\left(\binom{p+1}{i}+\binom{p+1}{i-1}\right)S_i(n)\\
&= \underbrace{n(n+1)^{p+1}+(n+1)\left((n+1)^p-1\right)}_{\text{by Lemma \ref{first-lemma}}} -\sum_{i=1}^{p-1}\binom{p+1}{i}S_i(n)\\
&\qquad -(p+1)S_p(n)-\sum_{i=1}^{p}\binom{p+1}{i-1}S_i(n)\\
&= n(n+1)^{p+1}+\underbrace{(p+1)S_p(n)}_{\text{by Definition \ref{Spn-definition}}}-(p+1)S_p(n) \\
& \qquad -\sum_{i=1}^{p}\binom{p+1}{i-1}S_i(n)\\
&=n\sum_{i=0}^{p+1}\binom{p+1}{i}n^i-\sum_{i=1}^{p}\binom{p+1}{i-1}S_i(n)\\
&=\sum_{i=1}^{p+2}\binom{p+1}{i-1}n^i-\sum_{i=1}^{p}\binom{p+1}{i-1}S_i(n)\\
&=n^{p+2}+(p+1)n^{p+1}-\sum_{i=1}^{p}\binom{p+1}{i-1}S_i(n-1)\\
&=n^{p+2}+(p+1)n^{p+1}-\sum_{k=1}^{n-1}\sum_{i=1}^{p}\binom{p+1}{i-1}k^i\\
&=n^{p+2}+(p+1)n^{p+1}-\sum_{k=1}^{n-1}\underbrace{k\left((k+1)^{p+1}-k^p(p+k+1)\right)}_{\text{by Lemma \ref{second-lemma}}} \\
&=n^{p+2}+(p+1)n^{p+1}-\sum_{k=1}^{n-1}k(k+1)^{p+1}+\sum_{k=1}^{n-1}k^{p+1}(k-1) \\
&\qquad +\sum_{k=1}^{n-1}k^{p+1}(p+2)\\
&=n^{p+2}+(p+1)n^{p+1}-(n-1)n^{p+1}+\sum_{k=1}^{n-1}(p+2)k^{p+1}\\
&=(p+2)\left(n^{p+1}+\sum_{k=1}^{n-1}k^{p+1}\right)\\
&=(p+2)\sum_{k=1}^{n}k^{p+1}. 
\end{align*}

Hence, $\ds S_p(n)=\sum_{k=1}^{n}k^p$ for all natural values of $p$ and $n$.
\end{proof}

\section{An identity involving Bernoulli numbers.} 
\label{Bernoulli-section}

Recall that Bernoulli numbers, $B_m$ for $m\in\Z$, $m\ge0$, are defined recursively as follows:

\begin{definition} 
\label{Bernoulli-numbers-definition}
Let $B_0=1$, and for each $m\ge1$,
$$\sum_{i=0}^{m}\binom{m+1}{i}B_i=0.$$
\end{definition}

\begin{remark}
\label{odd-B-remark}
For $m \ge 3$ odd, $B_m = 0$ \cite[p.\ 107]{Conway}.
\end{remark}

Gessel \cite{Gessel} proved that for any nonnegative integers $m$ and $n$, 
\begin{equation}
\label{Gessel-equation}
\sum_{i=0}^m \binom{m}{i}B_{n+i}=(-1)^{m+n}\sum_{j=0}^n\binom{n}{j}B_{m+j}.
\end{equation}

\begin{remark}
\label{Bernoulli-remark}
Setting $n=0$ in (\ref{Gessel-equation}) yields 
\begin{equation}
\label{Bernoulli-remark-equation}
\ds (-1)^mB_m=\sum_{i=0}^{m}\binom{m}{i}B_{i}=\sum_{i=0}^{m}\binom{m}{m-i}B_{m-i}.\end{equation} 
\end{remark}

\begin{theorem}
\label{Bernoulli-identity-theorem} 
For $m,k\in\Z$, $m \ge 1$, $0\le k \le m$, 
$$(-1)^{m-k}\binom{m}{k}B_{m-k}=\sum_{i=k}^{m}\binom{m}{i}\binom{i}{k}B_{m-i}.$$ 
\end{theorem}

\begin{proof}
Replacing $m$ by $m-k$ in (\ref{Bernoulli-remark-equation}) gives 
\begin{align*} 
\ds (-1)^{m-k}B_{m-k}&=\sum_{i=0}^{m-k}\binom{m-k}{m-k-i}B_{m-k-i} \\ 
& = \sum_{i=k}^{m}\binom{m-k}{m-i}B_{m-i}. \\
\end{align*}
Now multiplying both sides by $\displaystyle \binom{m}{k}$ yields 
\begin{align*} 
\ds (-1)^{m-k}\binom{m}{k}B_{m-k}
& = \sum_{i=k}^{m}\binom{m}{k}\binom{m-k}{m-i}B_{m-i}\\
& = \sum_{i=k}^{m}\frac{m!}{k!(m-k)!}\cdot \frac{(m-k)!}{(m-i)!(i-k)!}B_{m-i}\\
& = \sum_{i=k}^{m}\frac{m!}{i!(m-i)!}\cdot \frac{i!}{k!(i-k)!}B_{m-i}\\
& =\sum_{i=k}^{m}\binom{m}{i}\binom{i}{k}B_{m-i}. \\ 
\end{align*}

\end{proof}

\section{Symmetry of \texorpdfstring{$S_p(n)$}{Spn}.}
\label{Spn-symmetry-section}

Observe that all three of the polynomial formulas in (\ref{power-sum-formulas}) have $0$ and $-1$ as roots, and the second one also has $-\frac{1}{2}$ as a root. In fact, more is true: not only are the roots symmetric about $-\frac{1}{2}$, but the polynomials themselves have symmetry about $-\frac{1}{2}$.   
We will show that for each natural $p$, the sum of $p$-th powers of natural numbers from $1$ to $n$, $S_p(n)$, is symmetric about $-\frac{1}{2}$. 

\begin{theorem} 
\label{Spn-theorem}
For each $p\in\N$, $S_p(-(n+1))=(-1)^{p+1}S_p(n)$. Thus the graph of $S_p(n)$ is symmetric about the vertical line at  $-\frac{1}{2}$ if $p$ is odd, and symmetric about the point $\left(-\frac{1}{2},0\right)$ if $p$ is even.
\end{theorem}

\begin{proof}
We use the following Faulhaber formula \cite[p.\ 107]{Conway}:

$$S_p(n)=\frac{1}{p+1}\sum_{i=0}^{p}(-1)^i\binom{p+1}{i}B_in^{p+1-i}.$$
Note that with an index change, this formula is equivalent to

$$S_p(n)=\frac{1}{p+1}\sum_{i=1}^{p+1}(-1)^{p+1-i}\binom{p+1}{i}B_{p+1-i}n^i.$$
Now, 
\begin{align*}
S_p(-(n+1)) 
& = \frac{1}{p+1}\sum_{i=1}^{p+1}(-1)^{p+1-i}\binom{p+1}{i}B_{p+1-i}\left(-(n+1)\right)^i\\
& =\frac{1}{p+1}\sum_{i=1}^{p+1}(-1)^{p+1-i}\binom{p+1}{i}B_{p+1-i}(-1)^i\sum_{k=0}^{i}\binom{i}{k}n^k\\
& =\frac{1}{p+1}\sum_{i=1}^{p+1}\sum_{k=0}^{i}(-1)^{p+1-i}(-1)^i\binom{p+1}{i}\binom{i}{k}B_{p+1-i}n^k\\
& =\frac{1}{p+1}\left(\sum_{i=1}^{p+1}(-1)^{p+1}\binom{p+1}{i}\binom{i}{0}B_{p+1-i} \right.\\
& \qquad \left. +\sum_{k=1}^{p+1}\sum_{i=k}^{p+1}(-1)^{p+1}\binom{p+1}{i}\binom{i}{k}B_{p+1-i}n^k\right)\\
& =\frac{1}{p+1}\left((-1)^{p+1}\underbrace{\sum_{i=0}^{p}\binom{p+1}{i}B_i}_{=0 \text{ by Definition \ref{Bernoulli-numbers-definition}}} \right. \\
& \qquad +\sum_{k=1}^{p+1}(-1)^{p+1}\sum_{i=k}^{p+1}\binom{p+1}{i}\binom{i}{k}B_{p+1-i}n^k
\left) \phantom{\underbrace{\sum_{j=0}^{p}\binom{p+1}{j}B_j}_{\text{=0 by Definition 3.1}}} \right. \\
& =\frac{1}{p+1}\left(\sum_{k=1}^{p+1}(-1)^{p+1}\sum_{i=k}^{p+1}\binom{p+1}{i}\binom{i}{k}B_{p+1-i}n^k\right) \\ 
& = \frac{1}{p+1}\left(\sum_{k=1}^{p+1}(-1)^{p+1}\underbrace{(-1)^{p+1-k}\binom{p+1}{k}B_{p+1-k}}_{\text{by Theorem \ref{Bernoulli-identity-theorem} for }m=p+1}n^k \right)\\
& = (-1)^{p+1} S_p(n). 
\end{align*}

Thus if $p$ is odd, then $S_p(n)=S_p\left(-(n+1)\right)$, so the graph of $S_p(n)$ is symmetric about the vertical line at  $-\frac{1}{2}$, and if $p$ is even, then $S_p(n)=-S_p\left(-(n+1)\right)$, so the graph of $S_p(n)$ is symmetric about the point $\left(-\frac{1}{2},0\right)$.
\end{proof}

\begin{corollary}
\label{symmetry-corollary}
For each $p\in\N$, the roots of $S_p(n)$ are symmetric about $-\frac{1}{2}$. When $p$ is even, $S_p(n)$ has $-\frac{1}{2}$ as a root.
\end{corollary}

\section{Open problems.}
\label{open-problems-section}

As shown in section \ref{recursive-definition-section}, $S_p(n)$ is a polynomial in $n$ of degree $p+1$. Therefore it has $p+1$ complex roots, counting with multiplicity. Below are two open questions about these roots. 
\begin{enumerate}
\item How many distinct real roots does $S_p(n)$ have, and what are their multiplicities? 
\item Are there any patterns in the roots, both real and complex, in addition to the symmetry described in Corollary \ref{symmetry-corollary}? 
\end{enumerate}

\begin{section}{Acknowledgment.}
The authors would like to thank the College of Science and Mathematics at California State University, Fresno for supporting this work. The authors would also like to thank the anonymous reviewer for their valuable suggestions.
\end{section}

\bigskip

\hrule

\bigskip

\noindent 2010 \emph{Mathematics Subject Classification} Primary 11B83; Secondary 11B68, 11B37.\\
\emph{Keywords:} number theory, power sums, Bernoulli numbers.

\bigskip 

\hrule 

\bigskip
\noindent (Concerned with sequences \href{http://www.oeis.org/A027641}{A027641} and \href{http://www.oeis.org/A027642}{A027642}.) 

\end{document}